\documentclass[12pt]{article}
\usepackage[active]{srcltx}
\usepackage{amssymb,latexsym,amsmath,enumerate,verbatim,amsfonts,amsthm}
\usepackage{cite}
\numberwithin{equation}{section}
\newtheorem{theorem}{Theorem}[section]
\newtheorem{corollary}[theorem]{Corollary}

\newtheorem{lemma}{Lemma}[section]
\newtheorem{example}{Example}[section]

\theoremstyle{definition}
\newtheorem{definition}{Definition}[section]
\numberwithin{equation}{section}

\begin{document}
	\title{B tensors and tensor complementarity problems	\thanks{This  work was supported by
			the National Natural Science Foundation of P.R. China (Grant No.
			11571095, 11601134).}}

	\date{}
\author{Yisheng Song\thanks{Corresponding author. School of Mathematics and Information Science  and Henan Engineering Laboratory for Big Data Statistical Analysis and Optimal Control, Henan Normal University, XinXiang HeNan,  P.R. China, 453007. Email: songyisheng@htu.cn.}, \quad   Wei Mei\thanks{School of Mathematics and Information Science, Henan Normal University, XinXiang HeNan,  P.R. China, 453007. Email: 1017187432@qq.com}}
	\maketitle

	\begin{abstract}
		In this paper, one of our main  purposes is to prove the boundedness of solution set of tensor complementarity problem with B tensor such that the specific bounds only depend on the structural properties of tensor.  To achieve this purpose,  firstly, we present that each B tensor is strictly semi-positive and each B$_0$ tensor is  semi-positive. Subsequencely,  the strictly lower and upper bounds of different  operator norms are given for two positively homogeneous operators defined by B tensor. Finally, with the help of the upper bounds of different  operator norms, we show the strcitly lower bound of solution set of tensor complementarity problem with B tensor.  Furthermore, the upper bounds of spectral radius and $E$-spectral radius of B (B$_0$) tensor are obtained, respectively, which achieves our another objective. In particular, such the upper bounds only depend on the principal diagonal entries of tensors.  \vspace{3mm}
		
		\noindent {\bf Key words:}\hspace{2mm} $B$ tensor, Tensor complementarity problem, Spectral radius, Norm, Upper and lower bounds.
	\vspace{3mm}
	
	\noindent {\bf AMS subject classifications (2010):}\hspace{2mm} 47H15, 47H12, 34B10, 47A52, 47J10, 47H09, 15A48, 47H07
	\end{abstract}

	\section{Introduction}
	
As a natural extension of linear complementarity problem, the tensor complementarity problem is a new topic emerged from the tensor community. Meanwhile, such a problem is a special type of nonlinear complementarity problems. So the tensor complementarity problem seems to have similar  properties to the linear complementarity problem, and to have its particular properties other than ones of the classical nonlinear complementarity problem, and to have  some nice properites that depended on itself special structure. The notion of the tensor complementarity problem was used firstly by Song and Qi \cite{SQ2015,SQ2017}.  Recently, Huang and Qi \cite{HQ2017} formulated an $n-$person noncooperative game as a tensor complementarity problem and showed that a
Nash equilibrium point of the multilinear game is equivalent to a solution of the tensor complementarity problem.
By using specially structured properities of tensors,  the properities of the tensor complementarity problem have been well studied in the literatures. For example, see Song, Yu \cite{SY2016} and Song, Qi \cite{SQ2016} for strictly semi-positive tensors, Gowda, Luo, Qi and Xiu \cite{GLQX2015} and Luo, Qi and Xiu \cite{LQX} for Z-tensors, Ding, Luo and Qi \cite{DLQ2015} for  P-tensors, Wang, Huang, Bai \cite{WHB2016} for exceptionally regular tensors, Bai, Huang and Wang \cite{BHW2016} for strong P-tensors, Che, Qi, Wei \cite{CQW} for some special tensors, Huang, Suo, Wang\cite{HSW} for Q-tensors. Song and Qi \cite{SQ13}, Ling, He, Qi \cite{LHQ2015,LHQ2016}, Chen, Yang, Ye \cite{CYY} studied the tensor eigenvalue complementarity problem.

In past several decades, numerous mathematical works concerned with error bound analysis for the solution of linear complementarity problem by means of the special structure of the matrix $A$. For more details, see \cite{CLWV2015,D2011,DLL2012,DLL2013,GP2009,SW2013}. Recently, motivated by the study on error bounds for linear complementarity problem, Song, Yu \cite{SY2016} and Song, Qi \cite{SQ5} extended the error bounds results of the linear complementarity problem to the tensor complementarity problem with strictly semi-positive tensors. However, there are relatively few works in the specific upper or lower bounds of the tensor complementarity problem, which is a weak link in this topics.

In this paper, we will give the boundedness of solution set of tensor complementarity problem with B tensor. Moreover, we will present the specific lower bounds of such a problem, that only depend upon the structural property of B tensor.

To achieve the above goal, we need study the structured properities of tensor. Nowadays, miscellaneous structured tensors have been widely studied (Qi and Luo \cite{QL2017}), which is one of hot research topic. For more detail, see Zhang et al. \cite{ZQZ2014} and Ding et al.  \cite{DQW2013} for M-tensors, Song, Qi \cite{SQ2015} for $P$ $(P_{0})$ tensors and $B$ $(B_{0})$ tensors, Li and Li \cite{LL2015} for double $B$ tensors, Song, Qi \cite{SQ2014,SQ-2017} and Mei, Song \cite{MW2016} for Hilbert tensors. Recently, the concept of B tensor was first used by Song and  Qi \cite{SQ2015}. They gave many nice structured properties that is similar to ones of B matrices.   For more nice properties and applications of B matrices, see Pe\~{n}a \cite{JMP2001,JMP2003}. It is well known that each B matrix is a P matrix. However, the same conclusion only holds for even order symmetric B tensor \cite{YY2014}.  Qi and Song \cite{QS2014} showed taht an even order symmetric $B$ tensor is positive definite. Yuan and You \cite{YY2014} proved that a non-symmetric $B$ tensor is not $P$ tensor in general.  So there are many special properities of B tensors for further and serious consideration.

In section 3, we prove that each B tensor is strictly semi-positive and each B$_0$ tensor is  semi-positive. So, the solution set of tensor complementarity problem with B tensor is bounded. In order to presenting the specific lower bounds of such a problem,  in section 4, we give the strictly lower and upper bounds of different  operator norms  for two positively homogeneous operators defined by B tensor. By means of such the upper bounds, we establish the strcitly lower bound of solution set of tensor complementarity problem with B tensor.  Furthermore, we  achieve our another objective with the help of upper bounds of operator norms. That is, we obtain  the upper bounds of spectral radius and $E$-spectral radius of B (B$_0$) tensor, which only depend on the principal diagonal entries of tensors.

	\section{Preliminaries and basic facts}
	
	An $m$-order $n$-dimensional tensor (hypermatrix) $\mathcal{A} = (a_{i_1\cdots i_m})$ is a multi-array of real entries $a_{i_1\cdots i_m}\in\mathbb{R}$, where $i_j \in [n]=\{1,2,\cdots,n\}$ for $j \in [m]=\{1,2,\cdots,m\}$. The set of all $m$-order $n$-dimensional tensor is denoted  by $T_{m,n}$.  If the entries $a_{i_1\cdots i_m}$ are
	invariant under any permutation of their indices, then we call $\mathcal{A}$  a symmetric tensor, denoted by $S_{m,n}$. Let  $\mathcal{A} = (a_{i_1\cdots i_m})\in T_{m,n}$ and a vector $x=(x_{1},x_{2},\cdots,x_{n})^\top\in \mathbb{R}^{n}$. Then $\mathcal{A}x^{m-1}$ is a vector with its $i$th component defined by $$(\mathcal{A}x^{m-1})_{i}:=\sum_{i_{2},\cdots,i_{m}=1}^{n}a_{ii_{2}\cdots i_{m}}x_{i_{2}}\cdots x_{i_{m}}, \forall i\in [n]$$ and $\mathcal{A}x^{m}$ is a homogeneous polynomial of degree $m$, defined by $$\mathcal{A}x^{m}:=x^{T}(\mathcal{A}x^{m-1})=\sum_{i_{1},i_{2},\cdots,i_{m}=1}^{n}a_{i_{1}i_{2}\cdots i_{m}}x_{i_{1}}x_{i_{2}}\cdots x_{i_{m}}.$$
	
	For any $q\in \mathbb{R}^{n}$,  the tensor complementarity problem, denoted by $TCP(\mathcal{A},q)$, is to find $x\in \mathbb{R}^{n}$ such that \begin{equation}\label{TCP} x\geq0, q+\mathcal{A}x^{m-1}\geq0 \mbox{ and } x^{T}(q+\mathcal{A}x^{m-1})=0,\end{equation} or to show that no such vector exists.
	
		An n-dimensional $B$ matrix $B=(b_{ij})$ is a square real matrix with its entries satisfying that for all $i\in [n]$ $$\sum_{j=1}^{n}b_{ij}>0 \mbox{ and } \frac{1}{n}\sum_{j=1}^{n}b_{ij}>b_{ik}, i\neq k$$ As a natural extension of $B$ matrices. Song and Qi \cite{SQ2015} gave the definition of $B$ and $B_{0}$ tensor.

\begin{definition}\label{defn21} Let $\mathcal{B}=(b_{i_{1}\cdots i_{m}})\in T_{m,n}$. Then $\mathcal{B}$ is said to be
	\begin{itemize}
			\item[(i)]  a $B$ tensor if for all $i\in[n],\sum\limits_{i_{2},\cdots,i_{m}=1}^{n}b_{ii_{2}i_{3}\cdots i_{m}}>0$ and $$\frac{1}{n^{m-1}}(\sum_{i_{2},\cdots,i_{m}=1}^{n}b_{ii_{2}i_{3}\cdots i_{m}})>b_{ij_{2}j_{3}\cdots j_{m}} \mbox{ for all } (j_{2},j_{3},\cdots,j_{m})\neq (i,i,\cdots,i);$$
			\item[(ii)]  a $B_{0}$ tensor if for all $i\in[n],\sum\limits_{i_{2},\cdots,i_{m}=1}^{n}b_{ii_{2}i_{3}\cdots i_{m}}\geq 0$ and $$\frac{1}{n^{m-1}}(\sum_{i_{2},\cdots,i_{m}=1}^{n}b_{ii_{2}i_{3}\cdots i_{m}})\geq b_{ij_{2}j_{3}\cdots j_{m}} \mbox{ for all } (j_{2},j_{3},\cdots,j_{m})\neq (i,i,\cdots,i).$$
			\end{itemize}
\end{definition}

 For each $i\in [n]$,
let
\begin{equation}\label{eq22}\beta_i(\mathcal{B})=\max\{0,
b_{ij_2j_3\cdots j_m}; (j_2, j_3, \cdots, j_m)\ne (i, i, \cdots, i),
j_2,j_3,\cdots, j_m \in [n] \}.\end{equation}

\begin{lemma}\label{lem21} \emph{(Song and Qi \cite[Theorem 5.1, 5.2, 5.3]{SQ2015})} Let $\mathcal{B}=(b_{i_{1}\cdots i_{m}})\in T_{m,n}$. If $\mathcal{B}$ is  a $B$ tensor, then for each $i\in [n]$, \begin{itemize}
		\item[(i)] $ b_{ii\cdots i}>|b_{ij_{2}j_{3}\cdots j_{m}}|  \emph{ for all } (j_{2},j_{3},\cdots,j_{m})\neq (i,i,\cdots,i), \ j_{2},j_{3},\cdots,j_{m}\in [n]$;
			\item[(ii)] $\sum\limits_{i_2,\cdots, i_m=1}^{n} b_{ii_2i_3\cdots
				i_m}>n^{m-1}\beta_i(\mathcal{B})$;
			\item[(iii)] $b_{ii\cdots i}>\sum\limits_{b_{ii_2\cdots i_m}<0} |b_{ii_2i_3\cdots i_m}|$.
\end{itemize}			
If $\mathcal{B}$ is a B$_0$ tensor, then the above two inequalities hold with ``$>$'' being
			replaced by ``$\ge$''.
\end{lemma}

The concepts of tensor eigenvalues were introduced by Qi \cite{LQ1,LQ2} to the higher order symmetric tensors, and the existence of the eigenvalues and some applications were studied there. Lim \cite{LL} independently introduced real tensor eigenvalues and obtained some existence results using a variational approach.

\begin{definition}\label{defn22} Let $\mathcal{A}=(a_{i_{1}\cdots i_{m}})\in T_{m,n}$.
	\begin{itemize}
			\item[(i)]  A number $\lambda\in \mathbb{C}$ is called an eigenvalue of $\mathcal{A}$ if there is a nonzero vector $x$ such that
	\begin{equation}\label{eq23}
	\mathcal{A}x^{m-1}=\lambda x^{[m-1]}.
	\end{equation}
and $x$ is called  an eigenvector of $\mathcal{A}$ associated with $\lambda$. We call such an eigenvalue $H$-eigenvalue if it is real and has a real eigenvector $x$, and call such a real eigenvector $x$  H-eigenvector.
			\item[(ii)] A number $\mu\in \mathbb{C}$ is said to be an $E$-eigenvalue of $\mathcal{A}$ if there exists a nonzero vector $x$ such that
	\begin{equation}\label{eq24}
	\mathcal{A}x^{m-1}=\mu x(x^{T}x)^{\frac{m-2}{2}},
	\end{equation}
	and $x$ is called an $E$-eigenvector of $\mathcal{A}$ associated with $\mu$. It is clear that if $x$ is real, then $x$ is also real. In this case, $\mu$ and $x$ are called a $Z$-eigenvalue of $\mathcal{A}$ and a $Z$-eigenvector of $\mathcal{A}$, respectively.\\
			\end{itemize}
\end{definition}

We now give the
definitions of  (strictly) semi-positive tensors (Song and Qi \cite{SQ2017,SQ2016}).

\begin{definition} \label{defn23}
	Let $\mathcal{A}  = (a_{i_1\cdots i_m}) \in T_{m, n}$.   $\mathcal{A}$ is said to be\begin{itemize}
		\item[(i)]  semi-positive iff for each $x\geq0$ and $x\ne0$, there exists an index $k\in [n]$ such that $$x_k>0\mbox{ and }\left(\mathcal{A} x^{m-1}\right)_k\geq0;$$
		\item[(ii)]   strictly semi-positive iff for each $x\geq0$ and $x\ne0$, there exists an index $k\in [n]$ such that $$x_k>0\mbox{ and }\left(\mathcal{A} x^{m-1}\right)_k>0.$$
	\end{itemize}
\end{definition}

For $x\in \mathbb{R}^{n}$, it is well known that
\begin{equation}\label{eq25}
\|x\|_{\infty}:=\max\{|x_{i}|;i\in [n]\}  \mbox{  and  }  \|x\|_{p}:=(\sum_{i=1}^{n}|x_{i}|^{p})^{\frac{1}{p}}(p\geq 1)
\end{equation}
are two main norms defined on $\mathbb{R}^{n}$. Then for a continuous, positively homogeneous operator $T:\mathbb{R}^{n}\rightarrow \mathbb{R}^{n}$, it is obvious that
\begin{equation}\label{eq26}
\|T\|_{p}:=\max \limits_{\|x\|_{p}=1}\|T(x)\|_{p}  \mbox{  and  }  \|T\|_{\infty}:=\max \limits_{\|x\|_{\infty}=1}\|T(x)\|_{\infty}
\end{equation}
are two operator norms of T.
For  $\mathcal{A}\in T_{m,n}$, we may define a continuous, positively homogeneous operator $T_{\mathcal{A}}:\mathbb{R}^{n}\rightarrow \mathbb{R}^{n}$ by
\begin{equation}\label{TA}
	T_{\mathcal{A}}(x): =
	\begin{cases}
	\|x\|_{2}^{2-m}\mathcal{A}x^{m-1}    & x\neq 0. \\
	0  &  x=0.
	\end{cases}
	\end{equation}
 When $m$ is even, we may define another continuous, positively homogeneous operator $F_\mathcal{A} : \mathbb{R}^n \to \mathbb{R}^n$ by for any $x \in \mathbb{R}^n$,
\begin{equation} \label{FA}
F_\mathcal{A} (x) :=\left(\mathcal{A}  x^{m-1}\right)^{\left[\frac1{m-1}\right]}.
\end{equation} The
following upper bounds and properities of the operator norm were established  by
Song and Qi \cite{SQ2013,SQ5}.

\begin{lemma}\label{lem22} \emph{(Song and Qi \cite[Theorem 4.3]{SQ2013} and \cite[Lemma 3, Lemma 4]{SQ5})} Let $\mathcal{A}=(a_{i_{1}\cdots i_{m}})\in T_{m,n}$, Then
	\begin{itemize}
		\item[(i)]  $\|T_{\mathcal{A}}\|_{\infty}\leq \max \limits_{i\in [n]}\sum \limits_{i_{2},\cdots,i_{m}=1}^{n}|a_{ii_{2}\cdots i_{m}}|;$
		\item[(ii)] $\|F_{\mathcal{A}}\|_{\infty}\leq \max \limits_{i\in [n]}\left(\sum \limits_{i_{2},\cdots,i_{m}=1}^{n}|a_{ii_{2}\cdots i_{m}}|\right)^{\frac1{m-1}}$ if $m$ is even;
		\item[(iii)]  $\|T_{\mathcal{A}}\|_{p}\leq n^{\frac{m-2}{p}}\left(\sum \limits_{i=1}^{n}\left(\sum \limits_{i_{2},\cdots,i_{m}=1}^{n}|a_{ii_{2}\cdots i_{m}}|\right)^{p}\right)^{\frac{1}{p}};$
	\item[(iv)]  $\|F_{\mathcal{A}}\|_{p}	\leq\left(\sum\limits_{i=1}^{n} \left(\sum\limits_{i_{2},\cdots,i_{m}=1}^{n}|a_{ii_{2}\cdots i_{m}}|\right)^{\frac{p}{m-1}}\right)^{\frac1p}$ if $m$ is even.
	\end{itemize}
\end{lemma}

\section{B tensor is strictly semi-positive}

\begin{theorem}\label{thm31} Let $\mathcal{B}$ be an $m$-order $n$-dimensional $B$ tensor. Then $\mathcal{B}$ is strictly semi-positive.
\end{theorem}

\begin{proof} Suppose that $\mathcal{B}$ is not strictly semi-positive. Then there exists $x\geq0$ and $x\ne0$,  for all  $k\in  [n]$ with $x_k>0$  such that $$\left(\mathcal{B} x^{m-1}\right)_k\leq0.$$
Choose $x_i>0$  with $x_i\geq x_k$ for all $k\in  [n]$. Clearly, $\left(\mathcal{B} x^{m-1}\right)_i\leq0.$  Then we have
$$\aligned
0\geq	\left(\mathcal{B} x^{m-1}\right)_i&=\sum_{i_{2},i_{3},\cdots,i_{m}=1}^{n}b_{ii_{2}\cdots i_{m}}x_{i_{2}}x_{i_{3}}\cdots x_{i_{m}}\\
        &=b_{ii\cdots i}x_i^{m-1}+\sum_{b_{ii_{2}\cdots i_{m}}<0}b_{ii_{2}\cdots i_{m}}x_{i_{2}}x_{i_{3}}\cdots x_{i_{m}}\\ & \ \ \ \ +\sum_{b_{ii_{2}\cdots i_{m}}>0}b_{ii_{2}\cdots i_{m}}x_{i_{2}}x_{i_{3}}\cdots x_{i_{m}}\\
        &\geq b_{ii\cdots i}x_i^{m-1}+\sum_{b_{ii_{2}\cdots i_{m}}<0}b_{ii_{2}\cdots i_{m}}x_i^{m-1}\\ & \ \ \ \ +\sum_{b_{ii_{2}\cdots i_{m}}>0}b_{ii_{2}\cdots i_{m}}x_{i_{2}}x_{i_{3}}\cdots x_{i_{m}}\\
		&\geq\left(b_{ii\cdots i}+\sum_{b_{ii_{2}\cdots i_{m}}<0}b_{ii_{2}\cdots i_{m}}\right)x_i^{m-1}.
		\endaligned
		$$
By Lemma \ref{lem21} (iii), we obtain $$0<\left(b_{ii\cdots i}+\sum_{b_{ii_{2}\cdots i_{m}}<0}b_{ii_{2}\cdots i_{m}}\right)x_i^{m-1}\leq\left(\mathcal{B} x^{m-1}\right)_i\leq0,$$
a contradiction. Consequently, $\mathcal{B}$ must be strictly semi-positive.
	\end{proof}
Using the similar proof technique, it is easy to prove the following conclusion.

\begin{theorem}\label{thm32} Let $\mathcal{B}$ be an $m$-order $n$-dimensional $B_0$ tensor. Then $\mathcal{B}$ is  semi-positive.
\end{theorem}
Clearly, by Theorem 3.2 of Song and Yu \cite{SY2016} and Corollary 3.3 of Song and Qi \cite{SQ2017}, we easily obtain the following.
\begin{corollary}
\label{cor33} Let $\mathcal{B}$ be an $m$-order $n$-dimensional $B$ tensor. Then for each $q \in \mathbb{R}^n$, the tensor complementarity problem $TCP(\mathcal{B}, q)$ has always a solution. Furthermore, the solution set
of the $TCP(\mathcal{B}, q)$ is bounded  for each $q \in \mathbb{R}^n$.
\end{corollary}

\section{Boundedness about B tensors}

By means of the special structure of B tensor, now we present its upper bounds of the operator norm.

\begin{theorem}\label{thm41} Let $\mathcal{B}$ be an m-order n-dimensional $B$ tensor. Then
\begin{itemize}
			\item[(i)]  $ n^{\frac{m}{2}}\max_{i\in [n]}\beta_i(\mathcal{B})<n^{\frac{2-m}2}\max\limits_{i\in [n]}\sum\limits_{i_{2},\cdots,i_{m}=1}^{n}b_{ii_{2}\cdots i_{m}}\leq\|T_{\mathcal{B}}\|_{\infty}< n^{\frac{m}{2}}\max \limits_{i\in [n]}b_{ii\cdots i}$;
			\item[(ii)] $n^{\frac{mp-2}{2p}} \left(\sum\limits_{i=1}^{n}\left(\beta_i(\mathcal{B})\right)^p\right)^\frac1p<n^{\frac{2p-pm-2}{2p}}\left( \sum\limits_{i=1}^{n} \left(\sum\limits_{i_{2},\cdots,i_{m}=1}^{n}b_{ii_{2}\cdots i_{m}}\right)^p\right)^\frac1p\leq\|T_{\mathcal{B}}\|_{p}< n^{\frac{mp-2}{2p}}(\sum \limits_{i=1}^{n}b_{ii\cdots i}^{p})^{\frac{1}{p}}$ if $p\geq1,$
			\end{itemize}
where $\beta_i(\mathcal{B})=\max\{0,
b_{ij_2j_3\cdots j_m}; (j_2, j_3, \cdots, j_m)\ne (i, i, \cdots, i),
j_2,j_3,\cdots, j_m \in [n] \}$.
\end{theorem}
\begin{proof}
(i)  Choose $e=(1,1,\cdots,1)^\top$. Then $\|e\|_\infty=1$ and  $\|e\|_2=n^{\frac12}$, and hence by Lemma \ref{lem21} (ii), we have $$
\aligned
\|T_{\mathcal{B}}(e)\|_{\infty}&= \max_{i\in [n]}\left|\|e\|_{2}^{2-m}\sum_{i_{2},\cdots,i_{m}=1}^{n}b_{ii_{2}\cdots i_{m}}\right|\\
&=n^{\frac{2-m}2}\max_{i\in [n]}\sum_{i_{2},\cdots,i_{m}=1}^{n}b_{ii_{2}\cdots i_{m}}\\
&>n^{\frac{2-m}2} \max_{i\in [n]}n^{m-1}\beta_i(\mathcal{B})\\
&= n^{\frac{m}{2}}\max_{i\in [n]}\beta_i(\mathcal{B}).
\endaligned
$$
Consequently, $$ n^{\frac{m}{2}}\max_{i\in [n]}\beta_i(\mathcal{B})<n^{\frac{2-m}2}\max\limits_{i\in [n]}\sum\limits_{i_{2},\cdots,i_{m}=1}^{n}b_{ii_{2}\cdots i_{m}}\leq\|T_{\mathcal{B}}\|_{\infty}.$$

Now we show the right inequality. It follows from Lemma \ref{lem21} (i) together with the fact that $\|x\|_{1}\leq n\|x\|_\infty$ and $\|x\|_{2}\leq \sqrt{n}\|x\|_{\infty}$ that
 $$
 \aligned
 \|T_{\mathcal{B}}\|_{\infty}&=\max_{\|x\|_{\infty}=1}\|T_{\mathcal{B}}(x)\|_{\infty}
 = \max_{\|x\|_{\infty}=1}\max_{i\in [n]}\left|\|x\|_{2}^{2-m}\sum_{i_{2},\cdots,i_{m}=1}^{n}b_{ii_{2}\cdots i_{m}}x_{i_{2}}x_{i_{3}}\cdots x_{i_{m}}\right|\\
 &\leq \max_{\|x\|_{\infty}=1} n^{\frac{2-m}2}\|x\|_{\infty}^{2-m}\max_{i\in [n]}\left(\sum_{i_{2},\cdots,i_{m}=1}^{n}|b_{ii_{2}\cdots i_{m}}||x_{i_{2}}||x_{i_{3}}|\cdots |x_{i_{m}}|\right)\\
 &<  n^{\frac{2-m}2}\max_{\|x\|_{\infty}=1}\|x\|_{\infty}^{2-m}\max_{i\in [n]}\left(b_{ii\cdots i}\sum_{i_{2},\cdots,i_{m}=1}^{n}|x_{i_{2}}||x_{i_{3}}|\cdots |x_{i_{m}}|\right)\\
 &=n^{\frac{2-m}2} \max_{\|x\|_{\infty}=1}\|x\|_{\infty}^{2-m}\max_{i\in [n]}\left(b_{ii\cdots i}\left(\sum_{k=1}^{n}|x_{k}|\right)^{m-1}\right)\\
 &=n^{\frac{2-m}2} \max_{\|x\|_{\infty}=1}\|x\|_{\infty}^{2-m}\max_{i\in [n]}(b_{ii\cdots i}\|x\|_{1}^{m-1})\\
 &\leq n^{\frac{2-m}2}\max_{\|x\|_{\infty}=1}\|x\|_{\infty}^{2-m}(n\|x\|_\infty)^{m-1}\max_{i\in [n]}b_{ii\cdots i}\\
 &= n^{\frac{m}{2}}\max_{i\in [n]}b_{ii\cdots i}.
 \endaligned
 $$
		
	(ii)  Choose $y=(n^{-\frac1p},n^{-\frac1p},\cdots,n^{-\frac1p})^\top$. Then $\|y\|_p=1$ and  $\|y\|_2=n^{\frac{p-2}{2p}}$, and hence  by Lemma \ref{lem21} (ii), we have $$
	\aligned
	\|T_{\mathcal{B}}(y)\|_p^p&= \sum_{i=1}^{n}\left|\|y\|_{2}^{2-m}\sum_{i_{2},\cdots,i_{m}=1}^{n}b_{ii_{2}\cdots i_{m}}(n^{-\frac1p})^{m-1}\right|^{p}\\
	&=n^{\frac{(p-2)(2-m)}{2}} \sum_{i=1}^{n} n^{1-m}\left(\sum_{i_{2},\cdots,i_{m}=1}^{n}b_{ii_{2}\cdots i_{m}}\right)^p\\
	&>n^{\frac{(p-2)(2-m)}{2}} \sum_{i=1}^{n} n^{1-m}\left(n^{m-1}\beta_i(\mathcal{B})\right)^p\\
	&= n^{\frac{mp-2}{2}} \sum_{i=1}^{n}\left(\beta_i(\mathcal{B})\right)^p.
	\endaligned
	$$
	Consequently, $$ n^{\frac{mp-2}{2p}} \left(\sum\limits_{i=1}^{n}\left(\beta_i(\mathcal{B})\right)^p\right)^\frac1p<n^{\frac{2p-pm-2}{2p}}\left( \sum_{i=1}^{n} \left(\sum_{i_{2},\cdots,i_{m}=1}^{n}b_{ii_{2}\cdots i_{m}}\right)^p\right)^\frac1p\leq\|T_{\mathcal{B}}\|_{p}.$$
	
Next we show the right inequality. By Lemma \ref{lem21} (i) together with the fact that  $\|x\|_{q}\leq \|x\|_{p}\leq n^{\frac{1}{p}-\frac{1}{q}}\|x\|_{q} \mbox{ for } q>p $, we obtain
$$
 \aligned
 \|T_{\mathcal{B}}\|_{p}^{p}&=\max_{\|x\|_{p}=1}\|T_{\mathcal{B}}(x)\|_{p}^{p}\\
 &= \max_{\|x\|_{p}=1}\sum_{i=1}^{n}\left|\|x\|_{2}^{2-m}\sum_{i_{2},\cdots,i_{m}=1}^{n}b_{ii_{2}\cdots i_{m}}x_{i_{2}}x_{i_{3}}\cdots x_{i_{m}}\right|^{p}\\
 &\leq \max_{\|x\|_{p}=1}\|x\|_{2}^{(2-m)p}\sum_{i=1}^{n}\left(\sum_{i_{2},\cdots,i_{m}=1}^{n}|b_{ii_{2}\cdots i_{m}}||x_{i_{2}}||x_{i_{3}}|\cdots |x_{i_{m}}|\right)^{p}\\
 &< \max_{\|x\|_{p}=1}\|x\|_{2}^{(2-m)p}\sum_{i=1}^{n}\left(b_{ii\cdots i}\sum_{i_{2},\cdots,i_{m}=1}^{n}|x_{i_{2}}||x_{i_{3}}|\cdots |x_{i_{m}}|\right)^{p}\\
 &= \max_{\|x\|_{p}=1}\|x\|_{2}^{(2-m)p}\sum_{i=1}^{n}\left(b_{ii\cdots i}\left(\sum_{k=1}^{n}|x_{k}|\right)^{m-1}\right)^{p}\\
 &= \max_{\|x\|_{p}=1}\|x\|_{2}^{(2-m)p}\sum_{i=1}^{n}(b_{ii\cdots i}\|x\|_{1}^{m-1})^{p}\\
 &\leq \max_{\|x\|_{p}=1}\|x\|_{2}^{(2-m)p}(\sqrt{n}\|x\|_{2})^{(m-1)p}\sum_{i=1}^{n}b_{ii\cdots i}^{p}\\
 &= \max_{\|x\|_{p}=1}n^{\frac{(m-1)p}{2}}\|x\|_{2}^{p}\sum_{i=1}^{n}b_{ii\cdots i}^{p}\\
 &\leq \max_{\|x\|_{p}=1}n^{\frac{(m-1)p}{2}}(n^{\frac{1}{2}-\frac{1}{p}}\|x\|_{p})^{p}\sum_{i=1}^{n}b_{ii\cdots i}^{p}\\
 &= n^{\frac{mp-2}{2}}\sum_{i=1}^{n}b_{ii\cdots i}^{p}.
 \endaligned$$
The desired conclusions follow.
\end{proof}

\begin{theorem}\label{thm42} Let $\mathcal{B}$ be an $m$-order $n$-dimensional $B$ tensor. If $m$ is even, then
	\begin{itemize}
		\item[(i)]  $ n\left(\beta_i(\mathcal{B})\right)^{\frac1{m-1}}<\max\limits_{i\in [n]}\left(\sum\limits_{i_{2},\cdots,i_{m}=1}^{n}b_{ii_{2}\cdots i_{m}}\right)^{\frac1{m-1}}\leq\|F_{\mathcal{B}}\|_{\infty}\leq\max\limits_{i\in [n]}\left(\sum\limits_{i_{2},\cdots,i_{m}=1}^{n}|b_{ii_{2}\cdots i_{m}}|\right)^{\frac1{m-1}}\\<n\max\limits_{i\in [n]}b_{ii\cdots i}^{\frac1{m-1}}$;
		\item[(ii)] $n^{\frac{p-1}{p}} \left(\sum\limits_{i=1}^{n}\left(\beta_i(\mathcal{B})\right)^{\frac{p}{m-1}}\right)^\frac1p<\dfrac1{\sqrt[p]{n}}\left(\sum\limits_{i=1}^{n} \left(\sum\limits_{i_{2},\cdots,i_{m}=1}^{n}b_{ii_{2}\cdots i_{m}}\right)^{\frac{p}{m-1}}\right)^{\frac1p}\leq\|F_{\mathcal{B}}\|_{p}<n^{\frac{p-1}{p}} \left(\sum\limits_{i=1}^{n}b_{ii\cdots i}^{\frac{p}{m-1}}\right)^\frac1p$ ($p\geq1).$
	\end{itemize}
\end{theorem}
\begin{proof}
	(i)  Choose $e=(1,1,\cdots,1)^\top$. Then $\|e\|_\infty=1$, and hence by Lemma \ref{lem21} (ii), we have $$
	\aligned
	\|F_{\mathcal{B}}(e)\|_{\infty}&= \max_{i\in [n]}\left|\sum_{i_{2},\cdots,i_{m}=1}^{n}b_{ii_{2}\cdots i_{m}}\right|^{\frac1{m-1}}\\
	&=\max_{i\in [n]}\left(\sum_{i_{2},\cdots,i_{m}=1}^{n}b_{ii_{2}\cdots i_{m}}\right)^{\frac1{m-1}}\\
	&>n \max_{i\in [n]}\left(\beta_i(\mathcal{B})\right)^{\frac1{m-1}}.
	\endaligned
	$$
	So, by the definition of the operator norm,  we have $$ n\max_{i\in [n]}\left(\beta_i(\mathcal{B})\right)^{\frac1{m-1}}<\max_{i\in [n]}\left(\sum_{i_{2},\cdots,i_{m}=1}^{n}b_{ii_{2}\cdots i_{m}}\right)^{\frac1{m-1}}\leq \|F_{\mathcal{B}}\|_{\infty}.$$
	From Lemma \ref{lem22} (ii) and \ref{lem21} (i), it follows that  $$\|F_{\mathcal{B}}\|_{\infty}\leq\max_{i\in [n]}\left(\sum_{i_{2},\cdots,i_{m}=1}^{n}|b_{ii_{2}\cdots i_{m}}|\right)^{\frac1{m-1}}<\max_{i\in [n]}\left(\sum_{i_{2},\cdots,i_{m}=1}^{n}b_{ii\cdots i}\right)^{\frac1{m-1}}= n\max_{i\in [n]}b_{ii\cdots i}^{\frac1{m-1}}.$$

	(ii)  Choose $y=(n^{-\frac1p},n^{-\frac1p},\cdots,n^{-\frac1p})^\top$. Then $\|y\|_p=1$, and hence  by Lemma \ref{lem21} (ii), we have $$
	\aligned
	\|F_{\mathcal{B}}(y)\|_p^p&= \sum_{i=1}^{n}\left|\sum_{i_{2},\cdots,i_{m}=1}^{n}b_{ii_{2}\cdots i_{m}}(n^{-\frac1p})^{m-1}\right|^{\frac{p}{m-1}}\\
	&=\sum_{i=1}^{n} \frac1n\left(\sum_{i_{2},\cdots,i_{m}=1}^{n}b_{ii_{2}\cdots i_{m}}\right)^{\frac{p}{m-1}}\\
	&>\frac1n\sum_{i=1}^{n} \left(n^{m-1}\beta_i(\mathcal{B})\right)^{\frac{p}{m-1}}\\
	&= n^{p-1} \sum_{i=1}^{n}\left(\beta_i(\mathcal{B})\right)^{\frac{p}{m-1}}.
	\endaligned
	$$
	Consequently, $$ n^{\frac{p-1}{p}} \left(\sum\limits_{i=1}^{n}\left(\beta_i(\mathcal{B})\right)^{\frac{p}{m-1}}\right)^\frac1p<\frac1{\sqrt[p]{n}}\left(\sum_{i=1}^{n} \left(\sum_{i_{2},\cdots,i_{m}=1}^{n}b_{ii_{2}\cdots i_{m}}\right)^{\frac{p}{m-1}}\right)^{\frac1p}\leq\|F_{\mathcal{B}}\|_{p}.$$
	
	Next we show the right inequality. By Lemma \ref{lem21} (i) together with the fact that  $ \|x\|_1\leq n^{1-\frac{1}{p}}\|x\|_p \mbox{ for } p>1 $, we obtain
	$$
	\aligned
	\|F_{\mathcal{B}}\|_{p}^{p}&=\max_{\|x\|_{p}=1}\|F_{\mathcal{B}}(x)\|_{p}^{p}= \max_{\|x\|_{p}=1}\sum_{i=1}^{n}\left|\sum_{i_{2},\cdots,i_{m}=1}^{n}b_{ii_{2}\cdots i_{m}}x_{i_{2}}x_{i_{3}}\cdots x_{i_{m}}\right|^{\frac{p}{m-1}}\\
	&< \max_{\|x\|_{p}=1}\sum_{i=1}^{n}\left(b_{ii\cdots i}\sum_{i_{2},\cdots,i_{m}=1}^{n}|x_{i_{2}}||x_{i_{3}}|\cdots |x_{i_{m}}|\right)^{\frac{p}{m-1}}\\
	&= \max_{\|x\|_{p}=1}\sum_{i=1}^{n}\left(b_{ii\cdots i}\left(\sum_{k=1}^{n}|x_{k}|\right)^{m-1}\right)^{\frac{p}{m-1}}= \max_{\|x\|_{p}=1}\|x\|_1^{p}\sum_{i=1}^{n}b_{ii\cdots i}^{\frac{p}{m-1}}\\
	&\leq \max_{\|x\|_{p}=1}(n^{1-\frac1{p}}\|x\|_p)^p\sum_{i=1}^{n}b_{ii\cdots i}^{\frac{p}{m-1}}\\&= n^{p-1}\sum_{i=1}^{n}b_{ii\cdots i}^{\frac{p}{m-1}}.
	\endaligned$$
		The desired conclusions follow.
\end{proof}

\begin{theorem}\label{thm43} Let $\mathcal{B}$ be an $m$-order $n$-dimensional $B_{0}$ tensor. Then
	\begin{itemize}
		\item[(i)]  $ n^{\frac{m}{2}}\max_{i\in [n]}\beta_i(\mathcal{B})\leq\|T_{\mathcal{B}}\|_{\infty}\leq n^{\frac{m}{2}}\max \limits_{i\in [n]}b_{ii\cdots i}$;
		\item[(ii)] $n^{\frac{mp-2}{2p}} \left(\sum\limits_{i=1}^{n}\left(\beta_i(\mathcal{B})\right)^p\right)^\frac1p\leq\|T_{\mathcal{B}}\|_{p}\leq n^{\frac{mp-2}{2p}}\left(\sum \limits_{i=1}^{n}b_{ii\cdots i}^{p}\right)^{\frac{1}{p}}$ if $p\geq1;$
		\item[(iii)]  $ n\left(\beta_i(\mathcal{B})\right)^{\frac1{m-1}}\leq\|F_{\mathcal{B}}\|_{\infty}\leq n\max\limits_{i\in [n]}b_{ii\cdots i}^{\frac1{m-1}}$ if $m$ is even;
		\item[(iv)] $n^{\frac{p-1}{p}} \left(\sum\limits_{i=1}^{n}\left(\beta_i(\mathcal{B})\right)^{\frac{p}{m-1}}\right)^\frac1p\leq\|F_{\mathcal{B}}\|_{p}\leq n^{\frac{p-1}{p}} \left(\sum\limits_{i=1}^{n}b_{ii\cdots i}^{\frac{p}{m-1}}\right)^\frac1p$   if $m$ is even and  $p\geq1.$
	\end{itemize}
\end{theorem}

For a $B$ tensor $\mathcal{B}$, it is obvious that its upper bounds of the operator norm are simpler in form than ones of Lemma \ref{lem22}, but we cannot compare their size relationship. By constructing following two examples, we show that there exists  $B$ tensors such that the above upper bounds of $\|T_{\mathcal{B}}\|$ and $\|F_{\mathcal{B}}\|$ are smaller than the ones of Lemma \ref{lem22}.

\begin{example}\label{exa41} Let $\mathcal{B}=(b_{i_{1}i_{2}i_{3}i_{4}})\in S_{4,3}$, where $b_{1111}=b_{3333}=6$, $b_{2222}=5$, $b_{1333}=b_{3133}=b_{3313}=b_{3331}=1$, $b_{2322}=b_{2232}=b_{2223}=b_{3222}=1.5$ and all other $b_{i_{1}i_{2}i_{3}i_{4}}=2$. Then  $\mathcal{B}$ is a B tensor.
\end{example}
\begin{proof} It is obvious that $\sum \limits_{i_{2},i_{3},i_{4}=1}^3b_{ii_{2}i_{3}i_{4}}>0$ for $i=1,2,3.$
	
For $i=1$, we have
$$
\frac{1}{n^{m-1}}\left(\sum \limits_{i_{2},i_{3},i_{4}=1}^3b_{1i_{2}i_{3}i_{4}}\right)= \frac{1}{3^{4-1}}\times 57
= \frac{57}{27}>2,
$$
and hence,
$\frac{1}{n^{m-1}}\left(\sum\limits_{i_{2},i_{3},i_{4}=1}^3b_{1i_{2}i_{3}i_{4}}\right)>b_{1j_{2}j_{3}j_{4}}$ for all $(j_{2},j_{3},j_{4})\neq (1,1,1)$.

For $i=2$, we have
$$
\frac{1}{n^{m-1}}\left(\sum \limits_{i_{2},i_{3},i_{4}=1}^3b_{2i_{2}i_{3}i_{4}}\right)= \frac{1}{3^{4-1}}\times 55.5= \frac{55.5}{27}>2,
$$
 and so,
$\frac{1}{n^{m-1}}\left(\sum\limits_{i_{2},i_{3},i_{4}=1}^{n}b_{2i_{2}i_{3}i_{4}}\right)>b_{2j_{2}j_{3}j_{4}}$ for all $(j_{2},j_{3},j_{4})\neq (2,2,2)$.

For $i=3$, we have
$$
\frac{1}{n^{m-1}}\left(\sum \limits_{i_{2},i_{3},i_{4}=1}^3b_{3i_{2}i_{3}i_{4}}\right)= \frac{1}{3^{4-1}}\times 54.5= \frac{54.5}{27}>2,
$$
and so,
$\frac{1}{n^{m-1}}\left(\sum\limits_{i_{2},i_{3},i_{4}=1}^{n}b_{3i_{2}i_{3}i_{4}}\right)>b_{3j_{2}j_{3}j_{4}}$ for all $(j_{2},j_{3},j_{4})\neq (3,3,3)$.

Thus, $\mathcal{B}$ is a $B$ tensor.

 Clearly, for the upper bounds of $\|T_{\mathcal{B}}\|_\infty$, we  have
$$n^{\frac{m}{2}}\max \limits_{i\in [n]}b_{iiii}=3^{\frac{4}{2}}\times 6=54\mbox{ and }
\max \limits_{i\in [n]}\left(\sum_{i_{2},i_{3},i_{4}=1}^{n}|b_{ii_{2}i_{3}i_{4}}|\right)=57,$$
 and hence, $n^{\frac{m}{2}}\max \limits_{i\in [n]}b_{iiii}< \max \limits_{i\in [n]}\left(\sum \limits_{i_{2},i_{3},i_{4}=1}^{n}|b_{ii_{2}i_{3}i_{4}}|\right)$.

 It is obvious that for the upper bounds of $\|F_{\mathcal{B}}\|_p$ and $p=1$, $$3^{\frac{p-1}{p}} \left(\sum\limits_{i=1}^{3}b_{iii i}^{\frac{p}{m-1}}\right)^\frac1p<\left(\sum\limits_{i=1}^{3} \left(\sum\limits_{i_{2},i_3,i_4=1}^{3}b_{ii_{2}i_3 i_4}\right)^{\frac{p}{m-1}}\right)^{\frac1p}.$$

\end{proof}

\begin{example}\label{exa42} Let $\mathcal{B}=(b_{i_{1}i_{2}i_{3}i_{m}})\in S_{4,4}$, where $b_{1111}=b_{2222}=b_{3333}=b_{4444}=3$, $b_{1444}=b_{4144}=b_{4414}=b_{4441}=0.7$, $b_{2333}=b_{3233}=b_{3323}=b_{3332}=0.5$ and all other $b_{i_{1}i_{2}i_{3}i_{4}}=1$.
\end{example}
\begin{proof} It is obvious that $\sum \limits_{i_{2},i_{3},i_{4}=1}^4b_{ii_{2}i_{3}i_{4}}>0$ for $i=1,2,3,4.$

For $i=1$,
$$
\frac{1}{n^{m-1}}\left(\sum \limits_{i_{2},i_{3},i_{4}=1}^{n}b_{1i_{2}i_{3}i_{4}}\right)= \frac{1}{4^{4-1}}\times 65.7
= \frac{65.7}{64}>1.
$$
So we have
$\frac{1}{n^{m-1}}\left(\sum\limits_{i_{2},i_{3},i_{4}=1}^{n}b_{1i_{2}i_{3}i_{4}}\right)>b_{1j_{2}j_{3}j_{4}}$ for all $(j_{2},j_{3},j_{4})\neq (1,1,1)$.

For $i=2$,
$$
\frac{1}{n^{m-1}}\left(\sum \limits_{i_{2},i_{3},i_{4}=1}^{n}b_{2i_{2}i_{3}i_{4}}\right)= \frac{1}{4^{4-1}}\times 65.5= \frac{65.5}{64}>1.
$$
So we have
$\frac{1}{n^{m-1}}\left(\sum\limits_{i_{2},i_{3},i_{4}=1}^{n}b_{2i_{2}i_{3}i_{4}}\right)>b_{2j_{2}j_{3}j_{4}}$ for all $(j_{2},j_{3},j_{4})\neq (2,2,2)$.

For $i=3$,
$$
\frac{1}{n^{m-1}}\left(\sum \limits_{i_{2},i_{3},i_{4}=1}^{n}b_{3i_{2}i_{3}i_{4}}\right)= \frac{1}{4^{4-1}}\times 64.5= \frac{64.5}{64}>1.
$$
So we have
$\frac{1}{n^{m-1}}\left(\sum\limits_{i_{2},i_{3},i_{4}=1}^{n}b_{3i_{2}i_{3}i_{4}}\right)>b_{3j_{2}j_{3}j_{4}}$ for all $(j_{2},j_{3},j_{4})\neq (3,3,3)$.

For $i=4$
$$
\frac{1}{n^{m-1}}\left(\sum \limits_{i_{2},i_{3},i_{4}=1}^{n}b_{4i_{2}i_{3}i_{4}}\right)= \frac{1}{4^{4-1}}\times 65.1= \frac{65.1}{64}>1.
$$
So we have
$\frac{1}{n^{m-1}}\left(\sum\limits_{i_{2},i_{3},i_{4}=1}^{n}b_{4i_{2}i_{3}i_{4}}\right)>b_{4j_{2}j_{3}j_{4}}$ for all $(j_{2},j_{3},j_{4})\neq (4,4,4)$.

Therefore, $\mathcal{B}$ is a $B$ tensor.   It is obvious that for the upper bounds of $\|T_{\mathcal{B}}\|_p$,
$$n^{\frac{mp-2}{2p}}\left(\sum_{i=1}^{n}b_{iiii}^{p}\right)^{\frac{1}{p}}=4^{\frac{4p-2}{2p}}\times (3^{p}\times 4)^{\frac{1}{p}}=4^{\frac{2p-1}{p}}\times 3\times 4^{\frac{1}{p}}=48$$ and
$$
\aligned
n^{\frac{m-2}{p}}\left(\sum_{i=1}^{n}\left(\sum_{i_{2},i_{3},i_{4}=1}^{n}|b_{ii_{2}i_{3}i_{4}}|\right)^{p}\right)^{\frac{1}{p}}&=4^{\frac{4-2}{p}}\times(65.7^{p}+65.5^{p}+64.5^{p}+65.1^{p})^{\frac{1}{p}}\\
&>4^{\frac{4-2}{p}}\times(64^{p}\times 4)^{\frac{1}{p}}\\
&=64\times 4^{\frac{3}{p}}.
\endaligned
$$
Since $4^{\frac{3}{p}}>1$,  we have $n^{\frac{mp-2}{2p}}\left(\sum \limits_{i=1}^{n}b_{iiii}^{p}\right)^{\frac{1}{p}}< n^{\frac{m-2}{p}}\left(\sum \limits_{i=1}^{n}\left(\sum \limits_{i_{2},i_{3},i_{4}=1}^{n}|b_{ii_{2}i_{3}i_{4}}|\right)^{p}\right)^{\frac{1}{p}}$.

\end{proof}

\begin{theorem}\label{thm44} Let $\mathcal{B}$ be a  $B$ tensor. Then
\begin{itemize}
	\item[(i)] $|\lambda|  < \left(\sum\limits_{i=1}^{n}b_{ii\cdots i}^{\frac{1}{m-1}}\right)^{m-1}$  for all eigenvalues $\lambda$ of  $\mathcal{B}$ if $m$ is even;
	\item[(ii)] $|\mu|< n^{\frac{m}{2}}\min\left\{\max \limits_{i \in [n]}b_{ii\cdots i},\dfrac1{n}\sum\limits_{i=1}^nb_{ii\cdots i}\right\}$  for all $E$-eigenvalues $\mu$ of $\mathcal{B}$.
\end{itemize}
\end{theorem}
\begin{proof}
	(i) From the definition of the operator $F_\mathcal{B}$, it follows that $\lambda$ is an eigenvalue of $\mathcal{B}$ if and only if $\lambda^{\frac1{m-1}}$ is an eigenvalue of $F_\mathcal{B}$. By Theorem 4.2 of Song and Qi \cite{SQ2013}, we have
	$$|\lambda|^{\frac1{m-1}}\leq\|F_\mathcal{B}\|_1,\mbox{ i.e., } |\lambda|\leq\|F_\mathcal{B}\|_1^{m-1}.$$
By Theorem \ref{thm42} (ii), we have $$\|F_{\mathcal{B}}\|_1< \sum\limits_{i=1}^{n}b_{ii\cdots i}^{\frac{1}{m-1}}.$$	
So we obtain $$|\lambda|\leq\|F_\mathcal{B}\|_1^{m-1}<\left(\sum\limits_{i=1}^{n}b_{ii\cdots i}^{\frac{1}{m-1}}\right)^{m-1}.$$

(ii) From the definition of the operator $T_\mathcal{B}$, it follows that $\mu$ is an $E$-eigenvalue of $\mathcal{B}$ if and only if $\mu$ is an eigenvalue of $T_\mathcal{B}$. By Theorem 4.2 of Song and Qi \cite{SQ2013}, we have $$|\mu|\leq \|T_\mathcal{B}\|_\infty\mbox{ and }|\mu|\leq \|T_\mathcal{B}\|_1.$$
By Theorem \ref{thm41}, we obtain $$\|T_\mathcal{B}\|_\infty<n^{\frac{m}{2}}\max \limits_{i \in [n]}b_{ii\cdots i}\mbox{ and } \|T_\mathcal{B}\|_1<n^{\frac{m-2}{2}}\sum\limits_{i=1}^nb_{ii\cdots i}.$$ Then we have $$|\mu|<n^{\frac{m}{2}}\max \limits_{i \in [n]}b_{ii\cdots i}\mbox{ and } |\mu|<n^{\frac{m-2}{2}}\sum\limits_{i=1}^nb_{ii\cdots i}.$$ So the desired conclusion follows.
\end{proof}
Similarly, we easily show the following.

\begin{theorem}\label{thm45} Let $\mathcal{B}$ be a  $B_0$ tensor. Then
	\begin{itemize}
		\item[(i)] $|\lambda|  \leq  \left(\sum\limits_{i=1}^{n}b_{ii\cdots i}^{\frac{1}{m-1}}\right)^{m-1}$  for all eigenvalues $\lambda$ of  $\mathcal{B}$ if $m$ is even;
		\item[(ii)] $|\mu|\leq n^{\frac{m}{2}}\min\left\{\max \limits_{i \in [n]}b_{ii\cdots i},\dfrac1{n}\sum\limits_{i=1}^nb_{ii\cdots i}\right\}$  for all $E$-eigenvalues $\mu$ of $\mathcal{B}$.
	\end{itemize}
\end{theorem}
	
For the tensor complementarity problem with a B tensor $\mathcal{B}$, denoted by $TCP(\mathcal{B},q)$, we may show the bound of solution set of $TCP(\mathcal{B},q)$.

\begin{theorem} \label{thm46}
Let $\mathcal{B}$ be a $B$ tensor. Assume that $x$ be a non-zero
solution of $TCP(\mathcal{B},q)$ and $y_{+}=(\max\{y_{1},0\},\max\{y_{2},0\},\cdots,\max\{y_{n},0\})^\top$. Then
\begin{itemize}
	\item[(i)] $\dfrac{\|(-q)_+\|_\infty}{n^{m-1}\max \limits_{i\in [n]}b_{ii\cdots i}}<\|x\|_\infty^{m-1}$;
	\item[(ii)] $\dfrac{\|(-q)_+\|_2}{n^{\frac{m-1}{2}}\left(\sum \limits_{i=1}^{n}b_{ii\cdots i}^{2}\right)^{\frac{1}{2}}}<\|x\|_2^{m-1}$;
	\item[(iii)] $\dfrac{\|(-q)_+\|_m}{n^{\frac{(m-1)^2}{m}}\left(\sum \limits_{i=1}^{n}b_{ii\cdots i}^{\frac{m}{m-1}}\right)^{\frac{m-1}{m}}}<\|x\|_m^{m-1}$ if $m$ is even.
\end{itemize}
\end{theorem}
\begin{proof} Theorem \ref{thm31} implies that each $B$ tensor is strictly semi-positive. From Theoorem 5 of Song and Qi \cite{SQ5}, it follows that  \begin{equation}\label{eq41}
	\frac{\|(-q)_+\|_\infty}{ n^{\frac{m-2}2}\|T_{\mathcal{A}}\|_\infty}\leq\|x\|_\infty^{m-1}\mbox{ and }\frac{\|(-q)_+\|_2}{ \|T_{\mathcal{A}}\|_2}\leq\|x\|_2^{m-1}.
	\end{equation}		
By  Theorem \ref{thm41}, we have \begin{equation}\label{eq42}\|T_{\mathcal{A}}\|_\infty<n^{\frac{m}2}\max \limits_{i\in [n]}b_{ii\cdots i}\mbox{ and }\|T_{\mathcal{A}}\|_2<n^{\frac{m-1}{2}}\left(\sum \limits_{i=1}^{n}b_{ii\cdots i}^{2}\right)^{\frac{1}{2}}.\end{equation}
 Then combining  the  inequalities \eqref{eq41} and \eqref{eq42}, the conclusions (i) and (ii) are proved.

Similarly, we may show (iii). In fact, if $m$ is even, Theoorem 5 of Song and Qi \cite{SQ5} implies  \begin{equation}\label{eq43}\frac{\|(-q)_+\|_m}{ \|F_{\mathcal{A}}\|_m^{m-1}}\leq\|x\|_m^{m-1}.\end{equation}
By  Theorem \ref{thm42}, we have  \begin{equation}\label{eq44}\|F_{\mathcal{A}}\|_m<n^{\frac{m-1}{m}}\left(\sum \limits_{i=1}^{n}b_{ii\cdots i}^{\frac{m}{m-1}}\right)^{\frac{1}{m}}.\end{equation}
So the  desired conclusion follows by combining  the  inequalities \eqref{eq43} and \eqref{eq44}.
\end{proof}

\end{document}